\documentclass{amsart}
\usepackage{graphicx,amsmath,amssymb,amsfonts,amsthm}
\usepackage[all,cmtip]{xy}
\usepackage{hyperref}
\newtheorem{thm}{Theorem}[section]
\newtheorem{cor}[thm]{Corollary}
\newtheorem{lem}[thm]{Lemma}
\newtheorem{prop}[thm]{Proposition}
\theoremstyle{definition}
\newtheorem{defn}[thm]{Definition}
\newtheorem{example}[thm]{Example}
\theoremstyle{remark}
\newtheorem{rem}[thm]{Remark}
\numberwithin{equation}{section}

\newcommand{\psitilde}{\widetilde{\psi}}
\newcommand{\OO}{\mathcal{O}}
\renewcommand{\AA}{\mathbf{A}}
\newcommand{\Auts}{\underline{\mathrm{Aut}}}
\newcommand{\AL}{\mathrm{AL}}
\newcommand{\atilde}{\widetilde{a}}
\newcommand{\btilde}{\widetilde{b}}
\newcommand{\Pic}{\mathrm{Pic}}
\newcommand{\FT}{FT}

\newcommand{\pDer}{\mbox{$p$-}\mathrm{Der}}
\newcommand{\DI}{\mathrm{DI}}
\newcommand{\Sch}{\mathsf{Sch}}
\newcommand{\Set}{\mathsf{Set}}
\newcommand{\CRing}{\mathsf{CRing}}
\newcommand{\Fun}{\mathrm{Fun}}
\newcommand{\Jhat}{\widehat{J}}
\newcommand{\Gr}{\mathrm{Gr}}
\newcommand{\Spec}{\mathrm{Spec}}
\newcommand{\PP}{\mathbf{P}}
\newcommand{\KS}{\mathrm{KS}}
\newcommand{\FFbar}{\overline{ \mathbf{F}} }
\newcommand{\ZZpurcom}{W_{p,\infty}( \mathbf{F}_p)}

\newcommand{\vareps}{\varepsilon}
\newcommand{\Uhat}{\widehat{U}}
\newcommand{\AAhat}{\widehat{\mathbf{A}}}
\newcommand{\hattimes}{\widehat{\times}}
\newcommand{\J}{J_p}

\newcommand{\Xhat}{\widehat{X}}

\newcommand{\eps}{\varepsilon}

\newcommand{\Vhat}{\widehat{V}}
\newcommand{\xdot}{\dot{x}}
\newcommand{\ydot}{\dot{y}}
\newcommand{\fphi}{f^{\phi}{}}
\newcommand{\inclusion}{\hookrightarrow}

\renewcommand{\hbar}{\bar{h}}
\newcommand{\FFpbar}{\bar{\mathbf{F}}_p}

\newcommand{\isom}{\cong}
\newcommand{\fbar}{\bar{f}}

\renewcommand{\r}{\frac{\partial f}{\partial p}}
\newcommand{\ZZ}{\mathbf{Z}}
\newcommand{\D}{\mathcal{D}}

\newcommand{\ann}{\mathrm{ann}}

\newcommand{\unit}{\mathrm{unit}}
\newcommand{\Spf}{\mathrm{Spf}}

\title[Deligne-Illusie Classes I]{Deligne-Illusie Classes I: Lifted Torsors of Lifts of the Frobenius for Curves.}
\author{Taylor Dupuy}
\email{dupuy@math.ucla.edu}

\begin{document}

\maketitle

\begin{abstract}
 For curves of genus bigger than one we prove that Buium's first arithmetic jet spaces admit the structure of a torsor under some line bundle.
 This result lifts a known constructions in characteristic $p$ where the first $p$-jet space modulo $p$ is a sheaf under the Frobenius tangent sheaf (parametrizing Frobenius linear derivations).
 In particular we show there is a natural family of lifts of the Frobenius tangent bundle so that the first $p$-jet space (and hence higher order lifts of the Frobenius) form torsor a under this bundle. 
 
 The \u{C}ech cohomology classes associated to this torsor structure, which we call the Deligne-Illusie class, has strong analogies with the classical Kodaira-Spencer class from deformation theory.
\end{abstract}

\tableofcontents

\section{Introduction}
In this paper $p$ will always denote a prime.
We will let $\CRing$ denote the category of commutative rings with a unit, 
$\CRing_B$ denote the category of commutative rings over a base ring $B$, $\Sch_S$ or $\Sch_B$ be the category of schemes over a scheme $S$ or a ring $B$ and $\Set$ denote the category of sets.

For a ring $B$ over a ring of $p$-adic integers we will use the notation $B_n = B/p^{n+1}$.
We will use $\widehat{B}$ or $B^{\hat{p}}$ to denote $p$-adic completion $\varprojlim B/p^{n+1}$. 
For a scheme $Y$ over such a ring $B$ we will use the notation $Y_n$ for the reduction modulo $p^{n+1}$ i.e.  $Y:= Y \otimes_B B_n$.
We will let $\widehat{Y} = \varinjlim Y_n$ denote the $p$-formal completion of a scheme $Y$ over a $p$-adic ring $B$.

By a curve in $\Sch_B$ we will mean a scheme of relative dimension $1$.

\subsection{(Differential algebraic) Kodaira-Spencer classes}
Let $K$ be a characteristic zero field with a derivation $D: K \to K$. 
Let $X/K$ be a smooth scheme. 
Let $(U_i \to X)_{i\in I}$ be a Zariski affine open cover of $X$ such that $D_i: \OO(U_i) \to \OO(U_i)$ are lifts of the derivation $D$ on $K$.
We can then form the cohomology class
$$ \KS(X):= [D_i - D_j] \in H^1(X,T_{X/K}) $$
where $T_{X/K}$ denotes the relative tangent sheaf, whose sections are $K$-linear derivations on $\OO$.
The class $\KS(X) \in H^1(X,T_{X/K})$ is called the \textbf{Kodaira-Spencer} class.

For a $X/K$ a variety over a field with a derivation, one can define a twisted version of the tangent bundle $J^1(X/K,D) \to X$ whose local sections correspond to derivations lifting the derivation $D$ on the base.
The space $J^1(X/K,D)$ is called the \textbf{first jet space} of $X/K$.
\begin{thm}[\cite{Buium1994}, Proposition 2.5, page 65] \label{thm: geomdescent}
Let $X/K$ be a smooth variety over a field with a derivation $D$.
Suppose in addition that $K$ is algebraically closed.
The following are equivalent
\begin{enumerate}
\item $\KS(X) =0$ in $H^1(X, T_{X/K})$
\item $J^1(X/K,D) \cong T_{X/K} $ as $T_{X/K}$-torsors.
\item  There exists some $X' \in \Sch_{K^{D}}$ such that 
 $$ X \cong X' \otimes_{K^{D}}K. $$
Here $K^{D}$ denotes the field of constants
$$ K^{D} = \lbrace c \in K: D(c) =0 \rbrace.$$
\end{enumerate}
\end{thm}

The aim of this paper is to show that an arithmetic analog of this theorem exists in the case of curves over the $p$-adic ring $R=\widehat{\mathbf{Z}_p^{ur}}$ the $p$-adic completion of the maximal unramified extension of the $p$-adic integers.

In the arithmetic variant of Theorem \ref{thm: geomdescent} the first jet space $J^1(X/K,D)$ is replaced by the first arithmetic jet space of Buium.
Local sections of the first arithmetic jet space of a scheme correspond to local lifts of the Frobenius.

\subsection{Witt vectors}
We refer to Hazewinkel \cite{Hazewinkel2009} and for an introduction to Witt vectors.

We recall that the full ($p$-typical) witt vectors $W_{p,\infty}$ are a functor from rings to rings.
A basic property is that $W_{p,\infty}(\mathbf{F}_p) = \ZZ_p$ the $p$-adic integers.
For $k \subset \FFpbar$ the ring $W_{p,\infty}(k)$ complete discrete valuation rings with residue field $k$; it is a $p$-adic completion of an unramified extension of the $p$-adic integers.
The ring $W_{p,\infty}(\FFpbar)$ is isomorphic to $\widehat{\mathbf{Z}_p^{ur}} = \ZZ_p[\zeta; \zeta^n =1 , p\nmid n]^{ \widehat{} }$, the $p$-adic completion maximal unramified extension of the $p$-adic integers.
All of there rings have a unique lift of the Frobenius $\phi$ which is constant on $\ZZ_p$ and acts on roots of unity by $\zeta \mapsto \zeta^p$.

We also recall that the truncated ($p$-typical) witt vectors of length two $W_{p,1}$ are a functor from rings to rings where
for a ring $A$ we have $W_{p,1}(A) = A \times A$ as sets with addition and multiplication rules given by
\begin{eqnarray*}
 (x_0,x_1)(y_0,y_1) &=& (x_0y_0, x_0^p y_1 + y_0^p x_1 + p x_1y_1), \\
 (x_0,x_1) + (y_0,y_1) &=& (x_0 + y_0, x_1 + y_1 + C_p(x_0,y_0) ),\\
 C_p(S,T) &=& \frac{S^p + T^p - (S+T)^p }{p} \in \ZZ[S,T].
\end{eqnarray*}
This functor has the property that $W_{p,1}(\mathbf{F}_p) \cong \ZZ/p^2$.
The ideal $V_p( W_{p,1}(A) ) = \lbrace (0,a) : a\in A \rbrace$ has square zero for every ring $A$.

\subsection{$p$-derivations and lifts of the Frobenius}
Let $A$ be a ring and $B$ be an $A$-algebra.
Let $p$ be a prime number ($p$ will always denote a prime in this paper). 
A $p$-\textbf{derivation} from $A$ to $B$ is a map of sets $\delta:A\to B$ such that for all $a,b\in A$ we have
\begin{eqnarray*}
\delta(a+b) &=& \delta(a) + \delta(b) + C_p(a,b), \\
\delta(ab) &=& \delta(a)b^p + a^p \delta(b) + p\delta(a)\delta(b),\\
\delta(1) &=& 0.\\
C_p(S,T) &=& \frac{S^p + T^p - (S+T)^p }{p} \in \ZZ[S,T].
\end{eqnarray*}

These operations were introduced independently by Joyal \cite{Joyal1985} and Buium \cite{Buium1996}.
The collection of $p$-derivations from a ring $A$ to a ring $B$ will be denoted by $\pDer(A\to B)$.

\begin{example}
\begin{enumerate}
\item If $A=B = \ZZ_p$, the $p$-adic integers, then the map $\delta_p(x) = \frac{x-x^p}{p}$ defines a $p$-derivation.
\item If $A = \ZZ/p^2$ and $B= \ZZ/p$ then the division-by-$p$ map $[1/p]: p\ZZ/p^2 \to \ZZ/p$ makes sense and the map
 $\delta_p: \ZZ/p^2 \to \ZZ/p$ defined by $x\mapsto [1/p] (x - x^p)$ gives a $p$-derivation.
\end{enumerate}
\end{example}

For a ring $A$ we will let $W_{p,1}(A)$ denote the ring of $p$-typical Witt vectors of length two. 

A $p$-derivation $\delta: A\to B$ is equivalent to a map $A \to W_{p,1}(B)$ such that its composition with the canonical projection map $W_{p,1}(B)\to B$ is the underlying algebra map $A\to B$. 
This is similar to the fact that morphisms $A \to B[t]/(t^2)$ such that the composition with the projection $B[t]/(t^2) \to B$ give the algebra map $A \to B$, are equivalent to derivations from $A$ to $B$.

A \textbf{lift of the Frobenius} from $A\to B$ is a morphism $\phi: A\to B$ such that $$\phi(x) \equiv x^p \mod p.$$
If $B$ is a $p$-torsion free ring then a lift of a Frobenius is equivalent to a $p$-derivation and they are related by the formula $\delta_p(x) = \frac{\phi(x) - x^p}{p}$.

An expression for involving polynomial combinations of ring elements together with $p$-derivations will be called a Wittferential equation or arithmetic differential equation. 
A basic reference for this material is \cite{Buium2005}.

\subsection{Deligne-Illusie classes modulo $p$}
We will fix the following notation
\begin{itemize}
\item $R_0=k$ is a perfect field of characteristic $p$. 
\item $R = W_{p,\infty}(k)$ the ring of $p$-typical Witt vectors (equivalently, the $p$-adic completion of the maximal unramified extension of the $p$-adic integers).
\item $X/R$ a smooth scheme of finite type.
\item $\FT_{X_0}$ the $\OO_{X_0}$-module of Frobenius derivations.
For $D\in \FT_{X_0}$ a local section and $x,y\in \OO_{X_0}$ local sections we have 
\begin{eqnarray*}
D(xy) &=& D(x)y^p + x^p D(y), \\
D(x+y) &=& D(x) + D(y).
\end{eqnarray*}
Such derivations are called \textbf{Frobenius derivations}.
\end{itemize}

Let $\delta: R_1 \to R_{0}$ be the unique $p$-derivation from $R_1$ to $R_0$.
If $X/R$ is smooth, we can cover  $X$ by affine open subsets  $(U_i \to X_1)_{i\in I}$ and find local lifts of the $p$-derivations
$$ \delta_i: \OO(U_i)_1 \to \OO(U_i)_0.$$
The difference $\delta_i - \delta_j$ gives a well-defined map 
 $$ (\delta_i - \delta_j): \OO(U_{ij})_0  \to \OO(U_{ij})_0,$$ 
which is a derivation on the Frobenius, $(\delta_i - \delta_j) \in \FT_{X_0}( (U_{ij})_0)$.
The differences define a \u{C}ech cocycle for $\FT_{X_0}$ and one can check that the associated cohomology class is independent of the choice of lifts $\delta_i$.
Hence we have a well defined map
 $$ \DI_0: \pDer(R_1 \to R_0) \to H^1(X_0,\FT_{X_0}). $$
Since the $p$-derivation $R_1\to R_0$ is unique it will not hurt to denote the class associated to the lift $X_1$ by $\DI_0(X_1).$

Implicit in this construction is the fact that the sheaf $\pDer(\OO_{X_1} \to \OO_{X_0})$ is a torsor under $\FT_{X_0}$. 
We will say more about this in section \ref{sec: $p$-ders}.
The sheaf of $p$-derivations is representable, is called the first $p$-jet space of a curve modulo $p$, and will be denoted by $\J^1(X)_0$. 
This torsor appears many places in the literature under different names.
Sometimes it is refered to as ``the torsor of lifts of the Frobenius"  and is denoted by $\mathcal{L}$ in \cite{Ogus2007}.
The first $p$-jet space modulo $p$, $\J^1(X)_0$ is sometime known as the Greenberg transform $\Gr_1(X)$, this is the notation for example in \cite{Loeser2003}.

\begin{rem}
\begin{enumerate}
\item The construction of the Deligne-Illusie class is implicit in the proof of Theorem 2.1 in \cite{Deligne1987}. 
The class $\DI_0(X)$ is denoted by $c=[h_{ij}]$ in  \cite{Deligne1987}.
See in particular Remark 2.2.iii.
The construction also appears in \cite{Deligne1987} in the proof of Theorem 3.5 .

\item Chapter II, section 1, Theorem 1.1  in \cite{Mochizuki1996} also employs the Deligne-Illusie construction.
Implicit in the proof that $\D$ (the lifts of $X^{(p)}_0$, a Frobenius twist of $X$, to $\ZZ/p^2$) is a torsor under the first sheaf cohomology of the Frobenius tangent sheaf. 
We should note that in his treatment, Mochizuki considers schemes with log structures while we do not.
Mochizuki attributes the results in this section to \cite[proposition 4.12]{Kato1989} who attributes to \cite{Deligne1987}.
\item The Deligne-Illusie class $\DI_0(X_1)\in H^1(X_0,\FT_{X_0})$ should be compared to the classical deformation class $\KS(X_1)\in H^1(X_0, I_1\otimes T_{X_0})$ where $I_1$ is the ideal sheaf of $X_0 \hookrightarrow X_1$.
This construction of $\KS(X_1)$, in the equicharacteristic setting, can be found \cite{Olsson2007}. 
\end{enumerate}
\end{rem}

\subsection{Buium's arithmetic jet spaces}
Let $R = W_{p,\infty}(k)$ where $k$ is a perfect field of characteristic $p$.
Let $X/R$ be a scheme.
We define the \textbf{$r$th $p$-jet space functor} by
\begin{eqnarray*}
 &\J^{r}(X): \Sch_R \to \Set& \\
&\J^{r}(X)(A) = X(W_{p,r}(A))  \mbox{ for all } A \in \CRing_R. &
\end{eqnarray*}

The association $\J^{r}: \Sch_R \to  \Fun(\Sch_R, \Set)$ is functorial.
Here $\Fun$ denotes the category of functors where morphisms are natural transformations.
\begin{prop}[Borger \cite{Borger2011}, (12.5)]
For every $X/R$ a scheme of finite type, the functor $\J^r(X)$ is representable in the category of schemes.
\end{prop}

\begin{rem}
\begin{enumerate}
\item The functors we denote as $\J^r$ have been denoted as $W_{r*}$ by Borger in \cite{Borger2011}. 
\item In \cite{Buium1996} Buium proved that the functors $X \mapsto \J^{r}(X)_n$ are representable for every $n\geq 0$. 
Buium simply denotes these functors as $J^r(-)$.
\end{enumerate}
\end{rem}

It is important to know that local sections of the map $J_p^1(X)_n \to X_n$ correspond to local lifts of the Frobenius. 

\subsection{Statement of main result}

For $X \subset \PP^n_R$ a curve we will denote by condition $(*)$ the following
\begin{equation}
g(X) \geq 2  \mbox{ and } p> \deg(X \subset \PP^n_R).
\end{equation}

\begin{thm}[$\exists$ Lift of torsor of lifts of the Frobenius]
Let $X\subset \PP^n_R$ a smooth projective curve. 
If $X$ satisfies $(*)$ then there exists a system of ``canonically lifted'' $\OO_{X_n}$-modules $(\FT_{X_n})_{n\geq 0}$  such that 
the collection $(\J^{1}(X)_n)_{n\geq 0}$ form a system of torsors under $(\FT_{X_n})_{n \geq 0}$ compatible with the known structure in characteristic $p$.
For each $n\geq 0$ the torsor structure for $\J^1(X)_{n+1}$ under $\FT_{X_{n+1}}$ lifts the previous.
\end{thm}

For each $n\geq 0$ we define the \textbf{Deligne-Illusie} class
$$ \DI_n(X_n) \in H^1(X_n, \FT_{X_n})$$ 
to be the cohomology class associated to the torsor structure.

\begin{rem}
The classes $\DI_0(X)$ are known to exist for smooth $X/R$ of arbitrary dimension. 
The lifted classes are known to exists for abelian varieties.
Buium referes to these classes in \cite{Buium1995} and \cite{Buium2005} as Arithmetic Kodaira-Spencer classes and denotes them with $\KS$ instead of $\DI$.
See \cite{Buium2005}, Definition 3.10  for Deligne-Illusie classes for varieties in characteristic $p$ and \cite{Buium2005}, Definition 8.50 for a variant for Abelian varieties (which can also be constructed in characteristic zero).

In \cite[Lemma 4.4]{Buium1995}, Buium relates $\DI_0(A)$ of an abelian variety (denoted $\rho^{int}$ there) to $\KS(A_1/R_1)$ (denoted $\rho^{ext}$ and viewed as a map).
He proves 
$$\DI_0(A/R)=F^*\KS( [ \delta(t(A)) \mod p ]^{1/p} ),$$ 
where $F$ denotes the absolute Frobenius,
$t: R[[t_{ij}: 1\leq i, j \leq \dim_R(A) ]] \to R$ is the Serre-Tate classifying map for $A$ with image $t(A)$ 
and the bar denote reduction modulo $p$.
We refer to \cite{Buium1995} for more details.

After pairing a Deligne-Illusie class with elements of its Serre dual one can obtain arithmetic differential equations in the coefficients of the variety which is zero precisely when the variety admits a lifts of the Frobenius.
In the case that the variety under consideration is an elliptic curve, the resulting differential equation is a differential modular form (in the sense of Buium) which cuts out canonical lift on modular curves. 
See \cite[Section 3.9]{Buium2009} for an appearence in an application and \cite{Buium2000} for more on differential modular forms.

\end{rem}

\begin{defn}
The category of $\Lambda_p$-\textbf{Schemes} $\Sch^{\Lambda_p}_R$ (resp $\Sch^{\Lambda_p}_{R_n}$) is defined by
\begin{description}
\item[Objects] Schemes $X/R$ (resp $X_n/R_n$) with a lifts of the Frobenius on $R$ (resp $R_n$)
\item[Morphisms] Morphisms of schemes over $R_n$ equivariant with respect to the Frobeniuses.
\end{description}
\end{defn}
For $X' \in \Sch^{\Lambda_p}_{R}$ (resp $\Sch^{\Lambda_p}_{R_n}$) we will let $- \otimes_{\Lambda_p} R: \Sch^{\Lambda_p}_R \to \Sch_R$ denote the forgetful functor (resp $- \otimes_{\Lambda_p}R_n$).

\begin{cor}\label{thm: arithdescent}
The following are equivalent
\begin{enumerate}
\item $\DI_n(X_n) =0$ in $H^1(X_n, \FT_{X_n})$
\item $J_p^1(X)_n \cong \FT_{X_n}$ as a torsors under $\FT_{X_n}$.
\item  $X_n/R_n$ descends to the category of $\Lambda_p$-schemes: 
There exists some $X_n' \in \Sch^{\Lambda_p}_{R_n}$ such that $X_n' \otimes_{\Lambda_p} R_n = X_n$. \footnote{ This is just a fancy notation for saying that $X_n$ admits a lift of the Frobenius.}
\end{enumerate}
\end{cor}

\begin{rem}
\begin{enumerate}
\item Compare the statement to theorem \ref{thm: geomdescent}.
\item When $R = W_{p,\infty}(\FFbar_p)$ we have 
 $$R^{\delta_p} = \lbrace c \in R: \delta_p(c) =0 \rbrace  = \lbrace \zeta : \zeta^n =1, p\nmid n \rbrace \cup \lbrace 0 \rbrace $$
which is a monoid of roots of unity. 
It is unclear if there exists an interpretation of descent in algebro-geometric Categories from say \cite{Lorscheid2012a},\cite{Toen2009} or \cite{Pena2009}.
\item The result of Raynaud \cite{Raynaud1983} show that curves $X/R$ of genus $g\geq 2$ do not have lifts of the Frobenius.  
Hence curves $X/R$ satisfying $(*)$ do not have lifts of the Frobenius and act as ``non isotrivial" in our setting.
\end{enumerate}
\end{rem}

\subsection{Remarks on the proof}
Let $R$ be an arbitrary commutative ring.
A morphism of schemes $\pi: E\to X$ is called an $\AA^1_R$-\textbf{bundle} if there exists an open cover $(U_i \to X)_{i\in I}$ and isomorphisms $\psi_i: \pi^{-1}(U_i) \to U_i \times_R \AA^1_R$ which respect the projections down to $U_i$.
A collections of trivializations together with the isomorphisms will be called a \textbf{atlas}.

We will let $\Auts(\AA^1)$ denote the functor which associates to a scheme $U$ the opposite group of automorphisms of $\OO(U)[t]$ which we view as groups of polynomials under composition with coeffients in $\OO(U)$.
When we restrict $\Auts(\AA^1)$ to open subsets of a scheme it becomes a sheaf of groups.
 
Let $G \leq \Auts(\AA^1)$ be a subgroup and $\pi: E\to X$ an $\AA^1_R$ bundle. 
A $G$-\textbf{atlas} will be an atlas $(\psi, U_i \to X)_{i\in I}$ such that for all $i,j \in I$ we have $\psi_{ij} \in G(U_{ij})$.
A $G$-\textbf{structure} will be a maximal $G$-atlas. 
When this happens we call $G$ the \textbf{structure group} of the bundle.

Let $G$ and $H$ be subgroups of $\Auts(\AA^1)$. 
Let $\psitilde_{ij},\psi_{ij} \in G(U_{ij})$ be cocycles with respect to some cover $(U_i \to X)_{i\in I}$.
We say that $\psi_{ij}$ and $\psitilde_{ij}$ are $H$-\textbf{compatible} if there exists a collection of $\psi_i \in H(U_i)$ for $i\in I$ such that $\psi_i \psi_{ij} = \psitilde_{ij}\psi_j$.
When this is the case we write $\psi_{ij} \sim_H \psitilde_{ij}$.

Let $\AL_1$ denote the $a+bT$ subgroup of $\Auts(\AA^1)$.
\begin{lem}\label{rem: classes}
Let $\pi: E \to X$ be an $\AA^1_R$-bundle.  
The bundle $E\to X$ is a torsor under some line bundle $L\to X$ if and only if $\pi$ admits an $\AL_1$-structure.
\end{lem}
\begin{proof}
It is clear that an $\AL_1$-structure induces a line bundle under a torsor. 
The converse can be seen by considering the functor of points on $E$, imposing the obvious torsor structure on $E$ and checking that the definition is well-defined.
\end{proof}

\begin{rem}\label{classes}
The natural projection $\AL_1 = \OO_{X_n} \rtimes \OO_{X_n}^{\times} \to \OO^{\times}_{X_n}$ induces a map $H^1(X_n,\AL_1) \to H^1(X_n, \OO^{\times}) = \Pic(X_n)$.

If $\sigma_n \in H^1(X_n, \AL_1)$ is the class associated to the torsor on $J_p^1(X)_n$ then 
the image of that class under the natural map is the class of the lifted Frobenius tangent sheaf $[\FT_{X_n}] \in \Pic(X_n).$  
\end{rem}

To prove that $\J^1(X)_n$ has the structure of a torsor under $\FT_{X_n}$ it suffices to show that $J_{p}^1(X)_n$ admits an $\AL_1$-structure.
Here are the following reduction steps.
\begin{description}
\item[Step 1] Show that $\J^1(X)_n$ admits the structure of an $\AA^1_{R_n}$-bundle.
\item[Step 2] Show that $\J^1(X)_n$ admits $A_{n}$-structure.  (We will introduce subgroups $A_{n}, A_{n,d} \leq \Auts(\AA^1_{R_n})$ of ``automorphisms of bounded degree'' which play a key roll in the proof.).\footnote{ This step uses the hypothesis $\deg(X)<<p$.}
\item[Step 3] Show by induction on $n$ that $\J^{1}(X)_n$ admits an $A_{n,n}$-structure.
\item[Step 4] Show by induction on $d$ that $\J^{1}(X)_n$ admits an $A_{n,d}$-structure it admits a $A_{n+1,d-1}$ structure for $d\geq 2$. ($A_{n,1} \leq \AL_1(\OO_{X_n} ) )$ \footnote{ This step uses the hypotheses $g(X)\geq 2$.}
\end{description}

The first step is a theorem of Buium (section \ref{sec: step1})
The second step is where most of the work happens:
  we perform some local computations for transition maps for plane curves and extend these results to imply the existence of $A_n$ structures for $n\geq 1$.
This is done in section \ref{sec: an structs}.
The third and fourth steps are done simultaneously in section \ref{sec: step34} and uses a ``pairing'' between group and \u{C}ech cohomology.

Sections \ref{sec: $p$-ders} and \ref{sec: pjets} provide background on $p$-derivations and $p$-jet spaces.

\subsection*{Acknowledgements}
The author is indebted to Alexandru Buium for his guidance and encouragement.
We would also like to thank James Borger for reviewing an earlier version of this manuscript and giving many useful comments and suggestions.
Final preparations of this manuscipt occured at MSRI during the Spring of 2014. 

\section{$p$-derivations}\label{sec: $p$-ders}
The material for this section is standard and can be obtained from (say) \cite{Buium2005} and \cite{Buium1996} and contains no new information.
We provide this introduction here for convenience of the reader.
\subsection{$p$-derivations}
Let $A$ and $B$ be rings, with $B$ an $A$-algebra. A $p$-\textbf{derivation} $\delta_p:A \to B$ is a map of sets satisfying the following axioms
\begin{eqnarray*}
\delta_p(a+b) &=& \delta_p(a) + \delta_p(b) + C_p(a,b)\\
\delta_p(ab) &=& \delta_p(a)b^p + a^p \delta_p(b) + p\delta_p(a)\delta_p(b) \\
\delta_p(1) &=& 0 \\
C_p(x,y) &=& \frac{x^p + y^p - (x+y)^p}{p} \in \ZZ[x,y]
\end{eqnarray*}

The category of rings with $p$-derivations is called the category of $\Lambda_p$-\textbf{rings}. 

\subsection{Examples}
Let $A$ be a ring and $a\in A$. Recall that we have a well-defined morphism
$$ \left[ \frac{1}{a} \right]: aA \to A/\ann(a). $$
Here $\ann(a)$ denotes the annihilator ideal of $a$.
\begin{example}
$\delta: \ZZ/p^2 \to \ZZ/p$ given by $\delta(x) = (x - x^p)/p$
where we interpret $1/p$ as a map
 $$ \frac{1}{p}: p \ZZ/p^2 \to \ZZ/p.$$ 
\end{example}

\begin{example}
 If $R=W_{p,\infty}(k)$ with $k$ perfect of characterisic $p$ then $R$ has a unique lift of the Frobenius $\phi$ on it. It hence has a unique $p$-derivation $\delta(x) = (\phi(x) - x^p)/p$.
\end{example}

\begin{thm}\label{ideals}
 Let $R=W_{p,\infty}(k)$ where $k$ is a perfect field of characteristic $p$.
\begin{enumerate}
\item $\delta_p(p^n) = \frac{p^n - p^{np}}{p} = p^{n-1} \cdot \unit $
\item $\delta_p(p^n \cdot \unit )= p^{n-1} \cdot \unit $
\item $(p^n, \delta_p(p^n), \delta_p^2(p^n), \ldots, \delta^r(p^n))_R = (p^{n-r})_R$
\end{enumerate} 
\end{thm}

\begin{proof}
The first property is trivial.
The second property follows from the computation
\begin{eqnarray*}
\delta_p(p^s \cdot u) &-& \delta_p(p^s) u^p + p^{sp} \delta_p(u) + p \delta_p(p^s) \delta_p(u) \\
&=& p^{s-1}\cdot u^p + p^{sp} \delta_p(u) + p^{s}\cdot \unit \cdot \delta_p(u) \\
&=& p^{s-1}( \unit + p  \cdot \mathrm{ junk }).
\end{eqnarray*}

We prove the last property by induction on $r$. 
It is sufficient to show that $\delta_p^r(p^n) = p^{n-r} \cdot \unit$. 
We have
$$ \delta_p(\delta_p^{r-1}(p^n)) = \delta_p( p^{n-r+1} \cdot \unit ) = p^{n-r} \cdot \unit,$$ 
where the first equality follows from inductive hypothesis and the second equality follows from the second proposition.
\end{proof}

\subsection{First $p$-jet ring} \label{p-jet}
Define $(-)_{p,1}: \CRing \to \CRing$ by 
 $$ A_{p,1} = A[\dot{a}: a \in A]/(\mbox{relations})$$
where $(\mbox{relations})$ are generated by
\begin{eqnarray}
 \dot{(ab+c)} &=& \dot{a} b^p + a^p \dot{b} + p \dot{a} \dot{b} + \dot(c) + C_p(ab,c),\label{linear} \\
 C_p(x,y) &=& \frac{x^p + y^p - (x+y)^p}{p} \in \ZZ[x,y],
\end{eqnarray}
For all $a,b ,c \in A$.

\begin{rem}
Let $R = W_{p,\infty}(k)$ where $k \subset \FFpbar$. If $A$ is an $R$-algebra and $R$ admits multiple $p$-derivations we may want to impose that the $p$-derivation on $A$ extend the one on the base. 
Suppose $\delta_0: R\to R$ is such a $p$-derivation on the base. 
The additional relation we impose is then $\dot{r} = \delta_0(r)$ where of course these are understood to be taken as an image in $A$.

Since we will work modulo $p$th powers or $p$-formal setting in this paper, this will not matter.
\end{rem}

\begin{example}
$A/R$ is finite type,
$$ A = R[x_1,\ldots, x_n]/(f_1,\ldots,f_r) = R[x]/(f)$$
where $x = (x_1,\ldots, x_n)$ , $f = (f_1,\ldots,f_r)$ then
$$ A_{p,1} = R[x,\dot x]/(f, \dot f) $$
where $\dot x = (\dot x_1, \ldots, \dot x_n)$ and $\dot f = (\dot f_1, \ldots, \dot f_r)$. 
Here $\dot{(f_1)},\ldots, \dot{(f_r)} \in R[x,\dot x]$ are computed using the rule for linear combinations above (\ref{linear}).
\end{example}

\begin{thm}[Universal Property] \label{up affine}
There is a universal $p$-derivation $\delta_{p,1}: A \to A_{p,1}$ mapping $a$ to $\dot a$. 
It satisfies the following universal property:

For every $p$-derivation $\delta: A\to B$ of the ring homomorphism $A \to B$ there exists a unique ring homomorphism $u_{\delta}: A_{p,1} \to B$ such that
$$\xymatrix{
A\ar[r]^{\delta} \ar[dr]_{\delta_{p,1}} & B \\
& A_{p,1} \ar[u]_{u_{\delta}} 
}.  $$
The ring homomorphism is the morphism of $A$-algebras defined by $u_{\delta}(\dot{a}) = \delta(a)$. \footnote{ Warning: The diagram is not a diagram in the categorical sense but it is an exercise to show that the universal property can be formulated in terms of diagrams}
\begin{proof}
It is clear the the morphism is well-defined from the definitions in section \ref{p-jets}.
We leave it as an exercise to check the universality.
\end{proof} 
\end{thm}

\subsection{Data of $p$-derivations}

\begin{lem}[flatness over witt vectors= $p$-torsion free]

Let $A$ be an $R = W_{p,\infty}(k)$ algebra with $k$-perfect of characteristic $p$.
The following are equivalent 
\begin{enumerate}
\item $A$ is flat over $R$
\item The multiplication by $p$ morphism is injective.
\item  $A$ is $p$-torsion free
\end{enumerate} 
\end{lem}

\begin{proof}
It is clear the (2) and (3) are the same. 
We will show $p$-torsion free implies flat.
Flatness is equivalent to $I \otimes_R A \to I A$ given by $i \otimes_R a \mapsto i a$ is injective.
We have $I = m^n$ for some $m$ where $m = (p)$ is the maximal ideal of $R$.
A general element of $m^n \otimes A$ looks like $ \sum_i p^{n'+n_i} \otimes a_i$ with $n'$ be the the gcd of all of the $p^{n'+n_i}$ where we can assume wlog that $a_i$'s are not divisible by any powers of $p$.
Suppose $\sum_i p^{n'+n_i} \otimes a_i \mapsto p^{n'} (\sum_i p^{n_i} a_i) =0$.
Since multplication by $p$ is injective we have $\sum_i p^{n_i} a_i=0$.
This is a contradiction since $\sum_i p^{n_i} a_i$ was cooked up to be a unit.

We will show that flatness implies $p$-torsion free.
We prove the converse by contrapositive: If it is not $p$-torsion free it will not be flat.
Suppose that multiplication by $p$ is not injective on $A$.
This means that the map $pR \otimes_R A \to pA$ is not an injection.
This contradicts flatness.
\end{proof}

\begin{thm}
Let $B\in \CRing_A$, $A\in \CRing_R$ where $R = W_{p,\infty}(k)$ and $k$ is a perfect field of characteristic $p$.
Suppose that $A$ and $B$ are flat over $R$.
The following data are equivalent.

\begin{enumerate}
\item A $p$-derivation $\delta:A \to B$ of the algebra map $A\to B$.
\item An action $\rho: A \to W_{p,1}(B)$ (meaning a morphism of rings such that $(\pi_{p,1})_B \circ g) = f: A \to B$ the algebra map.
\item A morphism of $A$-algebras $A_{p,1} \to B$.
\end{enumerate}
\end{thm}

\begin{proof}
 Follows from the definitions.
\end{proof}

\begin{example}
Let $A$ and $B$ be rings over $R = W_{p,\infty}(k)$ with $k$ perfect of characteristic $p$. 
Suppose $p\neq 2$ and consider the diagram
$$\xymatrix{
A_{p,1} \ar[r]^u & B \\
A\ar[u]
}  $$
This induces $A \to B. $

If $f: A \to B$ is already given and $A = A/p^{n+1}$ then $(A)_{p,1} = (A)_{p,1}/p^{n}$. 
This follows from the fact that 
$$ \delta(p^n) = \frac{p^n - p^{np}}{p} = p^{n-1}(1 - p^{n(p-1)}) $$
when $p$ is not a unit. 

Hence we have a factorization
$$ \xymatrix{
A_{p,1} \ar[r]^u & B \\
A \ar[u] \ar[r]^{\pi_{p,1}^*} & A/p^n \ar[u]
},$$ 
although $f: A\to B$ may not factor through a reduction modulo $p^n$ in general.
 
\end{example}

\begin{thm}
Let $B$ be a $p$-torsion free ring and $\phi$ a lift of the Frobenius on $B$ inducing a lift of the Frobenius on $A_n = B/p^{n+1}$. 
This then induces a well-defined $p$-derivation
$$ \delta_p: A_n \to A_{n-1}.$$
\end{thm}
\begin{proof}
In general, given any $A$ and a lift of the Frobenius $\phi: A \to A$, one can try to define
$$\delta_p: A \to A/\ann(p) $$
via
$$ \delta_p(a) = (\left[ \frac{1}{p}\right] \circ g)(a) $$
where $g(a) = \phi(a) - a^p$, and $g: A \to pA$ at least.

The difficulty in defining $\delta_p$ comes from the equality
$$\left[ \frac{1}{p}\right ] (g(a) g(b)) = p \cdot \left [ \frac{1}{p} \right](g(a)) \cdot \left[\frac{1}{p} \right] (g(b) )  \mbox{  in  }  A/ \ann(p). $$

We leave it to the reader to verify that this makes sense.

It is useful for the reader to note that if $A_n = B/p^{n+1}$ where $B$ is $p$-torsion free then 
\begin{eqnarray*}
\ann_A(p^j) &\cong& p^{n-j} A\\
A/\ann_A(p^j) &\cong& A/p^{n-j} 
\end{eqnarray*}
These give maps $[1/p]: pA_n \to A_{n-1}$. 
\end{proof}

\begin{thm}\label{local lifts}
Let $A,B$ be flat over $R=W_{p,\infty}(k)$ where $k\subset \FFpbar$. 
Suppose that $A$ is of finite type over $R$. 
Let $f: A\to B$ be a morphism of rings inducing the morphism of rings $f_n:A_n\to B_n$.
The following are equivalent
\begin{enumerate}
\item  \label{ch1} A lift of the Frobenius $\phi_n:A_n \to B_n$,
 $$ \phi_n(a) \equiv f_0(a)^p \mod p $$
\item \label{ch2} A $p$-derivation $\delta_p: A_n \to B_{n-1}$
\item \label{ch3} A morphism $(A_{p,1})_{n-1} \to B_{n-1}$ of $A_{n-1}$-algebras. 
\end{enumerate}
 
\end{thm}

\begin{proof}
To see that \ref{ch2} implies \ref{ch1} note that $\phi_n(a) := a^p + p \delta(a)$ defines a lift of the Frobenius.
We will show that \ref{ch3} and \ref{ch2} are equivalent: Let $A = R[x]/(f)$ so that $(A_{p,1})_n = (R[x,\dot{x}]/(f,\dot{f}))/p^n = R_{n-1}[x,\dot{x}]/(f,\dot{f})$.
The map clearly defines a $p$-derivation. (Note: $ (\delta_{p,1} )_n: A_n \to (A_{p,1})_{n-1}$ is universal).

We will not show \ref{ch1} implies \ref{ch2} but the reader can verify that this follows from the universal property of $p$-derivations.  
\end{proof}

\begin{lem}
 Let $A$, $B$ and $C$ be flat $R=W_{p,\infty}(k)$-algebras where $k \subset \FFpbar$.
 Suppose $A \to B$ is an \`{e}tale morphism of rings.
 Every $p$-derivations $B_n \to C_{n-1}$ lifts to a unique $p$-derivation $A_n \to C_{n-1}$.
\end{lem}
\begin{proof}
 The proof is essentially the same as the standard proof for lifting infinitesimal deformations.
 We prove a stronger result from which our result follows a fortiori.
 
 Recall that \`{e}tale ring homomorphisms have the infinitesimal lifting property: For every commutative diagram
 \begin{equation}\label{inf}
  \xymatrix{ 
  A\ar[r]\ar[d]^{\alpha} \ar[d] & B\ar[d]^{\beta} \\
  C \ar[r] & C/I }, \ \ \ I^2 =0, 
 \end{equation}
 there exists a unique map $\widetilde{\beta}:B\to C$ making the diagram commute.
 
 We want to show that when we are given a $p$-derivation
 $$\xymatrix{
   A \ar[d] \ar[dr]& \\
   W_{p,1}(C') \ar[r] & C'
   }$$
 There exists a diagram 
 $$\xymatrix{
   B \ar[d] \ar[dr]& \\
   W_{p,1}(C') \ar[r] & C'
   }$$
 lifting the previous.
 
 To apply the infinitesimal lifting criterion (\ref{inf}) with the following choices:
   \begin{eqnarray*}
    C &=& W_{p,1}(C'), \\
    C/I &=& C', \\
    I &=& V_p(W_{p,1}(C')), \\
    \alpha &=& \mbox{ map assoc. to $p$-der $B \to C'$},\\
    \beta &=& \mbox{ alg map $A \to C$ }.
   \end{eqnarray*}
\end{proof}

\section{$p$-Jets}\label{sec: pjets}
References for this section include \cite{Buium2005}, \cite{Buium1996} and \cite{Borger2011}.
Other than in presentation, this section contains no new information.

We summarize the results of this section:
\begin{enumerate}
\item Suppose $X/R=W_{p,\infty}(\FFpbar)$ is flat. 
Then local sections of the map $(\pi_{p,1})_n: J_p^1(X)_n \to X_n$ induce local $p$-derivations/lifts of the Frobenius on $X_n$ on the ring $A$ and conversely.
\item If $X/R$ is flat then $\J^{r}(X_n) = \J^r(X)_{n-r}$, in particular $\J^1(X_n) = \J^1(X)_{n-1}$.
\item Suppose $X/R$ is flat, we have the following compatibility between $p$-jet functors and open and closed immersions.
\begin{eqnarray*}
J_{p}^r(\mbox{ open immersion} )_n =  \mbox{ open immersion}.\\
J_{p}^r(\mbox{ closed immersion} )_n = \mbox{ closed immersion } 
\end{eqnarray*}
\end{enumerate}

\subsection{$p$-jet spaces}
Let $X/R$ be a scheme where $R = W_{p,\infty}(k)$, with $k$ perfect of characteristic $p$.
define the $r$th $p$-\textbf{jet functor} $J_p^r(X): \CRing_R \to \Set$ to be the functor of $W_{p,r}$ valued points of $X$:
$$ J_p^r(X)(A) := X(W_{p,r}(A)) \ \ \ \ A \in \CRing_R.$$

The natural morphism of ring schemes $\pi_{r,s}: W_{p,r} \to W_{p,s}$ for $r>s$ induces functorial morphisms $J_p^r(X) \to J_{p}^s(X)$.
The morphisms $\pi_r: W_{p,r} \to \OO$ induce functorial morphisms $J_p^r(X) \to X$. 

\begin{example}
 When $X = \Spec(A)$ and $A$ is an $R$ algebra with $R = W_{p,\infty}(k)$ where $k$ is perfect of characteristic $p$ we have that $J_p^1(X)$ is representable and
 $$ J_p^1( \Spec(A) ) = \Spec( A_{p,1} )$$
as schemes over $X$. 
\end{example}

\begin{rem}
 Since the constuction $A \mapsto A_{p,1}$ does not localize well one needs to work hard to get that $p$-jet spaces are representable.
 This essentially follows from the quotient rule for $p$-derivations:
  $$ \delta \left(\frac{1}{f} \right ) = \frac{f^p \delta(f)}{f^p(f^p + p \delta(f))}.$$
\end{rem}

For $X/R=W_{p,\infty}(k)$ flat where $k\subset \FFpbar$, we define the sheaf of $\OO_{X_n}$-algebras $\OO_{X_n}^{(1)}$ to be the sheaf associated to presheaf
 $$ U \mapsto \OO(U)_{p,1} \mod p^{n+1}, $$
for relevant open subsets of $U$.
We will construct the global spectrum of this ring in order to produce the first $p$-jet spaces.
For higher order $p$-jets one does something similar.

\begin{thm}[\cite{Buium1996}]\label{buium thm}
Let $R = W_{p,\infty}(k)$ where $k$ is perfect of characteristic $p$.
\begin{enumerate}
\item Let $X/R$ be a flat scheme. 
The functor $J_p^r(X):= X \circ W_{p,r}$ over $X$, is representable when reduced modulo $p^{n+1}$ for every $n$.
\item Furthermore for every $A$ in $\CRing_{R_n}$ we have 
$$ J_p^r(X)_n(A) = J_p^r(X)(A) = X(W_{p,r}(A)) \to X(A) = X_n(A) $$
where the map is $(\pi_{p,r})_A$.
\end{enumerate}
\end{thm}

A more difficult theorem of Borger proves the following:

\begin{thm}[\cite{Borger2011}]
Let $X/R$ be any scheme and $R = W_{p,\infty}(k)$ with $k$ perfect of characteristic $p$. 
The functor $J_p^r(X) := X \circ W_{p,r}$ is representable in the category of schemes over $R$.
\end{thm}

\begin{thm}
Let $X/R$ be a scheme which is flat over $R$ (so that multiplication by $p$ is injective). 
\begin{enumerate}
\item The natural morphism $\pi_{m,s}: J_{p}^m(X_n) \to J_{p}^s(X_n)$ factors through reduction modulo $p^{n-m+1}$,
$$ \xymatrix{ J_{p}^m(X)_{n-m} \ar[r]^{(\pi_{p,r+s,s})_{n-r-s} } & J_{p}^s(X)_{n-m} }.$$
This is a morphism of schemes over $R_{n-m}$.
\item Local sections of the morphisms 
$$ \xymatrix{ J_p^1(X)_m \ar[r]^{\pi_{p,1}} & \ar[l]^{s} X_n } $$
are in bijection with local lifts of the Frobenius/$p$-derivations 
$$ \delta: \OO(X_{n+1}) \to \OO(X_n) = \OO(X_{n+1})/p^{n+1}. $$
\end{enumerate}
\begin{proof}
The problem is local.
Let $X = \Spec( R[x]/(f))$ (using multi-index notation).
The map $\pi_{m,s}$ gives a map of rings
 $$ R[x, \dot{x}, \ldots, x^{(s)}]/(f, \dot{f}, \ldots, f^{(s)}) = \OO(J_{p}^s(X)) \to \OO(J_{p}^m(X)) = R[x,\dot{x},\ldots, x^{(m)}]/(f,\dot{f},\ldots, f^{(m)}).$$
The first part of the proposition follows from an explicit description of the ideals given previously (in Theorem \ref{ideals}).
The second part follows from the characterization of lifts of the Frobenius on rings of the form $A_n = B/p^{n+1}$ (in \ref{local lifts}).
\end{proof} 
\end{thm}

\begin{thm}
 
Let $X/R$ be flat where $R = W_{p,\infty}(k)$ where $k$ is perfect and characteristic $p$. 
\begin{enumerate}
\item If $i: U \hookrightarrow X$ is an open immersion of $R$-schemes of finite type then 
$$ J_p^1(i)_n: J_p^1(U)_n\hookrightarrow J_p^1(X)_n $$
is an open immersion.
\item If $j: Z \hookrightarrow X$ is a closed immersion of $R$-schemes of finite type then 
$$ J_p^1(j): J_p^1(Z) \hookrightarrow J_p^1(X) $$
is also a closed immersion.
\end{enumerate}
\end{thm}

\begin{proof}
We first show the open immersion property is local for the functors $J_p^1(-)_n$.
Consider the case when $X = \Spec(A)$ with $A = R[x]/(f)$ (using multi-index notation). 
It is enough to show that the functor $J_p^1( )_n$ respects principal open immersions (one globalizes this affine result in the usual way taking direct limits of principal open immersions).
\begin{eqnarray*}
(A_{p,1})_g &=& ( R[x,\dot{x}]/(f,\dot{f}) )_g \\
(A_g)_{p,1} &-& R[x,\dot x, 1/g \dot{ (1/g)}]/(f,\dot{f})]\\
&=& R[x,\dot{x},1/g, \frac{-\dot{g} }{g^{2p} } \sum_{j\geq 0} \left( \frac{p\dot g}{g^p}\right)^j]/(f,\dot{f}) 
\end{eqnarray*}
so we clearly have 
$$ (A_{p,1})_g \hookrightarrow (A_g)_{p,1}. $$
Reducing modulo $p^{n+1}$ gives
$$ \frac{- \dot g }{g^{2p}} \sum_{j\geq 0} \left( \frac{-p \dot{g} }{g^p} \right)^j \in ( (A_{p,1})_g )_n. $$
This shows 
$$ ( (A_{p,1})_g )_n = ( (A_g)_{p,1})_n. $$
In the non-$p$-formal setting this result fails. 
We should also note that this is really the key observation construction of $p$-jet (Theorem \ref{buium thm}).

In the non-$p$-formal, non-affine case one needs to do more work.
For the full proof we refer the reader to \cite{Borger2011}.

In the affine setting 
\begin{eqnarray*}
X = \Spec \ R[x]/(f), & J_p^1(X) = \Spec \ R[x, \dot x]/(f,\dot f), \\
Z = \Spec \ R[x]/(f,g) & J_p^1(X) = \Spec \ R[x, \dot x]/(f,g, \dot f, \dot g).
\end{eqnarray*}
and it is clear that $\dot f$ and $\dot g$ give extra elements of the ideal. 
\end{proof}
We remark that the above theorem is true for higher order finite jets as well and the contructions are analogous. 

The above proposition implies that $J_p^1(Y)_n = (\pi_{p,1}^{-1})_n( Y_n)$ if $Y \hookrightarrow X$ is an open or closed immersion of $R$-schemes when $X$ is flat over $R$.

\subsection{$p$-Formal schemes}
The construction of $p$-jet spaces associated to a scheme $X/R$ where $R = W_{p,\infty}(\FFpbar)$ gives a system of maps
$$\xymatrix{
 & \vdots \ar[d]  & \vdots \ar[d] & \\
\cdots \ar[r]  & J_p^r(X)_n  \ar[r] \ar[d] & \ar[r] J_p^r(X)_{n-1} \ar[d] & \cdots  \\
\cdots \ar[r]  & J_p^r(X)_n \ar[r] \ar[d] &  \ar[r] J_p^r(X)_{n-1} \ar[d] & \cdots  \\
 & \vdots & \vdots  & 
}$$
The $p$-formal schemes $\widehat{J}_{p}^r(X):= \varinjlim \Jhat_p^r(X)_n$ used by Buium (in say \cite{Buium2005}) behave nicely. 
In some sense this means that the appropriate place for $p$-jet spaces would be some variant of the $p$-adic rigid analytic spaces. 
We make use of these limits in the subsequent sections.

\subsection{Examples}
We give some examples that we believe clarify the situation.

\begin{example}
Let $R = W_{p,\infty}(k)$ where $k$ is a perfect field of characteritic $p$. 
Let $X = \Spec( R[x]/(f))$ (using multi-index notation). 
There are no sections of the morphism of $R$-schemes
$$ \xymatrix{ 
J_p^1(X)_0 = J_p^1(X_1) \ar[r]^{\pi_{p,1}} & \ar[l] X_1 }. $$

This would correspond to a map of rings
$$ s^*: R[x,\dot{x}]/(f, p^2, \dot{f}, \dot{(p^2)}  ) = R_0[x,\dot{x}]/(f,\dot{f})  \to R_1[x]/(f) = R[x]/(f,p^2). $$ 
\end{example}

\begin{example}
Let $R = \ZZpurcom$. 
 Write 
  $$ \PP^1_R = \frac{\Spec( R[x]) \cup \Spec( R[y] )}{\sim }$$
where $\sim$ denotes gluing along $\Spec \ R[x,y]/(xy-1)$.
Then 
 $$\Jhat_p^1(\PP^1_R) = \frac{\Spf( R[x, \xdot]^{\widehat \ } ) \cup \Spf( R[y,\ydot]^{\widehat \ } )}{\sim} $$
where $\sim$ denotes gluing of the formal schemes along 
$$ \Spf( R[x,\xdot, y,\ydot]^{\widehat \ }/(xy-1,\xdot y^p + \ydot x^p + p \xdot \ydot) .$$
  
\end{example}
  
\section{Proofs}

\subsection{Step 1: Affine bundle structures}\label{sec: step1}

Recall that any smooth scheme $X/R$ of relative dimension $d$ admits a Zariski affine open cover opens $(U_i \to X)_{i\in I}$ such that there exist \'etale maps $f_i: U_i \to \AA^d_R$. 
Recall the following lemma:
\begin{lem}[\cite{Buium2005}, Section (3.2) ] \label{Lemma:trivialization lemma}
 Let $X$ and $Y$ be finite dimensional smooth schemes over $R = W_{p,\infty}(\FFpbar)$,
 \begin{enumerate}
  \item If $f: X\to Y$ is \'{e}tale then $\Jhat_p^1(X) \cong \Xhat \hattimes_Y \Jhat^1(Y)$.
  \item If $f: X\to \AA^d_R$ is \'etale then $\Jhat_p^1(X) \cong \Xhat \hattimes \AAhat^d_R$
 \end{enumerate}
\end{lem}

\begin{rem}\label{sec:trivializations_of_jet_spaces}
If $X = \Spec A$ and $f^*: R[T_1,\ldots,T_n]\to A$ is \'etale then $\OO(J^1(X)) = \OO(X)[\dot T_1,\ldots,\dot T_n]^{\widehat \ }$. 
Here we have identified the \'{e}tale parameters $T_i$ with their image under $f^*$.
\end{rem}

\subsection{Projections} \label{projs}
By a decomposition of $\PP^n$ (over $R$) we will mean a collection of linear forms $\lambda = \lbrace l_0,\ldots, l_n\rbrace$ in general position together with its associated linear subspaces.
For $0\leq d \leq n$ we will let $\lambda_r$ denote the collection of hyperplanes generated by $\lambda$ of dimension $d$.
For each such linear subspace $\Lambda$ we will let $\Lambda'$ denote is complementary subspace and $\pi^{\Lambda'}_{\Lambda}$ denote the linear projection onto $\Lambda$ with center $\Lambda'$.
For a linear subspace $\Lambda$ of $X$ and a point $x$ of $X$ we will let $\overline{x,\Lambda}$ denote the linear subspace spanned by $\Lambda$ and all the lines passing through points of $\Lambda$ and $x$.
Complementary subspaces have the property that $\overline{x,\Lambda} \cap \overline{x, \Lambda'} = \overline{x,\pi^{\Lambda}_{\Lambda}(x)}$ is the unique line passing through $x$ and the point of its projection.
We will denote this line by $L(\Lambda,\Lambda',x)$.
For a given $X \subset \PP^n$ and a complementary pair of subspace $\Lambda,\Lambda'$ we will let $X_{\Lambda}$ denote the open subset of $X$ where $\pi^{\Lambda'}_{\Lambda}$ restricted to $X$ is \`etale onto its image.

\subsection{Step 2: Existence of an $A_{n}$-structure} \label{sec: an structs}.

\label{sec:An a group}
Let $R = \ZZpurcom$. 
We define a subset of automorphisms of degree $n$ mod $p^{n}$
$$ A_n := \lbrace a_0 + a_1 T + p a_2 T^2 + \cdots + p^{n-1} a_n T^n : a_1 \in \OO_{X_n}^{\times}, a_i\in \OO_{X_n} \rbrace \subset \Auts( \AA^1_{R_{n-1}}). $$
\begin{prop}
$A_n \subset \Auts(\AA^1_{R_{n-1}})$ is a subgroup.
\end{prop}
\begin{proof}
We will first show that $A_n$ is closed under composition and then show that $A_n$ is closed under taking inverses. 
Let
\begin{eqnarray*}
f(T) &=& a_0 + a_1 T + p a_2 T^2 + \cdots + p^{n-1} a_n T^n, \\
g(T) &=& b_0 + b_1 T + p b_2 T^2 + \cdots + p^{n-1} b_n T^n  
\end{eqnarray*}
be elements of $A_n$. We claim that $g(f(T)) \in A_n$.  

If is sufficient to show that every term in 
     $$p^{j-1} b_j ( f(T) )^j, 1<j\leq n-1$$
of degree $d$ is divisible by $p^{d-1}$.

A typical term in the expansion above takes the form
    $$A=p^{j-1} \cdot (p^{i_1-1}a_{i_1} T^{i_1}) \cdots (p^{i_j-1}a_{i_j} T^{i_j}),$$
has degree greater than $d$. This means that $i_1 + i_2 +... + i_j = d$ and that 
   $p^{d-j} = p^{i_1 + i_2 + ... + i_j - j}$
which means that $A$ is of the form 
    $A=p^{d-1} a_{i_1} ... a_{i_j}  T^d$
and that every coefficient $T^d$ in the expansion of $g(f(T))$ is divisible by $p^{d-1}$. In particular note that $g(f(T))$ has degree $n$ mod $p^n$ which shows that $A_n$ is closed under composition. 

We will now show that if $f \in A_n$ then $f^{-1} \in A_n$. Fix $f(T) = a_0 + a_1 T + p a_2 T^2 + \cdots + p^{n-1} a_n T^n$. We proceed by induction on $n$. The base case is $n=2$ we have proved everything. Now suppose that 
    $$f(g(T)) = g(f(T)) = T \mod p^n$$
we need to show that $g \in A_n$. By induction we know that we can write (by rearranging terms if necessary) 
    $$g(T) = g_{n-1}(T) + p^{n-1} G(T)$$
where $G(T)$ has order greater than $n$ and 
   $$g_{n-1}(T) = b_0 + b_1 T + p b_2 T^2 + \cdots + p^{n-2} b_{n-1} T^{n-1}.$$
We will assume that $G(T)$ has degree greater strictly greater that $n$ and derive a contradiction. Examining 
   $$g(f(T)) = g_{n-1}(f(T)) + p^n G(f(T)) \mod p^n$$
we know from the previous proposition that
   $$\deg(g_{n-1}(f(T)))\leq n.$$
We also know that 
  $$p^{n-1} G(f(T)) = p^{n-1} G(a_0 + a_1 T)$$
and that the degree of $G(f(T))$ is exactly the degree of $G(T)$ since $a_1$ is a unit. This means that $g(f(T)) = T \mod p^n$ has degree strictly greater than $n$ which is a contradiction. This shows that $g(T)$ actually has degree $n$ and hence $g(T)\in A_n$ which completes the proof.
\end{proof}

In what follows we let $f_x$ and $f_y$ denote the usual partial derivatives of $f$ with respect to $x$ and $y$ respectively.

\begin{lem}[Local Computations]\label{lem:affine_transition} \label{lem:non-vanishing criteria}
Let $C=V(f)$ be a plane curve over $R=\ZZpurcom$ with $f\in R[x,y]$. Let $U=D(f_x)$ and $V=D(f_y)$ and $\eps_U$ and $\eps_V$ be the \'{e}tale projections to the $y$ and $x$ axes of $\AA^2$ and let $\psi_U: \J^1(U)\to \Uhat \hattimes \AAhat^1$ and $\psi_V: J^1(V)\to \Vhat \hattimes \AAhat^1$ be the associated affine bundle trivializations\footnote{ see Lemma \ref{Lemma:trivialization lemma}}. 

\begin{enumerate}
 \item If $f_x$ or $f_y$ is not identically zero modulo $p$ then 
 the transition map $\psi_{VU} := \psi_V \circ \psi_U^{-1}$ has the property that 
  $$\psi_{UV} \otimes_R R_n \in A_n$$
for each $n\geq 2$.
\item If $\deg(f)< p$ then $f_x$ or $f_y$ is not identically zero modulo $p$.
\end{enumerate}
\end{lem}

\begin{proof}
Assume without loss of generality that $f_y\neq 0 \mod p$. The maps $\varepsilon_U:U\to \mathbb A^1$ given by $(x,y)\mapsto y$ and $\varepsilon_V: V \to \mathbb A^1$ given by $(x,y)\mapsto x$ are \'etale. On these open sets we have $\mathcal O^1(U)= O(U)[\dot y]^{\hat \ }$ and $\mathcal O^1(V) = \mathcal O(V)[\dot x]^{\widehat p}$. This means we have 
 $$\OO(J^1(U\cap V)) = \OO(U\cap V)^1 = \mathcal O(U\cap V)[\dot x]^{\widehat p} = \mathcal O(U\cap V)[\dot y]^{\widehat p}.$$ 
Let $\psi_U:J^1(U)\to\Uhat\hattimes \AAhat^1$ be given by $t \mapsto \dot y$ and $\psi_V: J^1(V)\to \Vhat \hattimes \AAhat^1$ be given by $t \mapsto \dot x $. We can compute the transition map $\psi_V\circ \psi_U^{-1}\in \Auts(\AAhat^1)(U\cap V)$ by first computing what $\dot y$ is in terms of $\dot x$. We first have
	\begin{eqnarray*}
	\delta f &\equiv& \frac{1}{p}[f^{\phi}(x^p,y^p) - f(x,y)^p] + \nabla f^{\phi}(x^p,y^p)\cdot(\dot x,\dot y) \\
	&&  + \frac{p}{2}[ f_{xx}^{\phi}(x^p,y^p) \dot x^2 + 2 \fphi_{xy}(x^p,y^p)\dot x\dot y + f^{\phi}_{yy}(x^p,y^p)\ydot^2] \\
	&\equiv& 0 \mod p^2  \mbox{ in } \OO(U\cap V)[\ydot]^{\widehat \ }
	\end{eqnarray*}
where for a polynoimal $g(x) = a_0 + a_1x + \cdots + a_n x^n$ the polynomial $g^{\phi}(x) := \phi(a_0) + \phi(a_1)x + \cdots + \phi(a_n) x^n$ as usual and $\nabla f= (f_x,f_y)$ is the usual gradient from calculus.

Let 
\begin{eqnarray*}
 A &=& R + \fphi_x(x^p,y^p) \xdot + p\fphi_{xx}(x^p,y^p) \xdot^2/2,\\
 B &=& \fphi_y(x^p,y^p)+p\fphi_{xy}(x^p,y^p)\xdot,\\
 C &=& \fphi_{yy}(x^p,y^p)/2,\\
 R &=& (\fphi(x^p,y^p)-f(x,y)^p)/p
\end{eqnarray*}
then, solving the equation $A+ B\ydot + C \ydot^2 =0$ gives 
\begin{eqnarray*}
  \ydot = -\frac{A}{B}+p\frac{A^2C}{B^3}.
\end{eqnarray*}
Since
\begin{eqnarray*}
pB^{-3}A^2C &=& p\frac{(R+\fphi_x(x^p,y^p)\xdot )^2 \fphi_{yy}(x^p,y^p)}{2\fphi_y(x^p,y^p)^3}\\
AB^{-1} &=& \frac{1}{\fphi_y(x^p,y^p)} [ R+\fphi_x(x^p,y^p) \xdot + p \fphi_{xx}(x^p.y^p) \xdot^2/2 \\
		&&- p \frac{\fphi_{xy}(x^p,y^p)\xdot}{\fphi_y(x^p,y^p)}(R+\fphi_x(x^p,y^p) \xdot) ]
\end{eqnarray*}
we get
\begin{equation}\label{eqn:affine_transition}
 \ydot = \alpha + \beta \xdot + p\gamma \xdot^2
\end{equation}
where
\begin{eqnarray*}
 \alpha &=& -\frac{R}{\fphi_y(x^p,y^p)} + p \frac{R^2 \fphi_yy(x^p,y^p)}{2\fphi_y(x^p,y^p)^3}\\
 \beta &=& -\frac{-\fphi_x(x^p,y^p)}{\fphi_y(x^p,y^p)} + p \frac{\fphi_{xy}(x^p,y^p) R}{\fphi_y(x^p,y^p)^2}+\frac{pR\fphi_x(x^p,y^p)\fphi_yy(x^p,y^p)}{\fphi_y(x^p,y^p)^3}, \\
 \gamma &=& -\frac{\fphi_{xx}(x^p,y^p)}{2\fphi_y(x^p,y^p)}+ \frac{\fphi_{xy}(x^p,y^p)\fphi_x(x^p,y^p)}{\fphi_y(x^p,y^p)^2} + \frac{\fphi_x(x^p,y^p)^2 \fphi_{yy}(x^p,y^p)}{2\fphi_y(x^p,y^p)^3}.
\end{eqnarray*}

We will now show that $\ydot \equiv a_0 + a_1 \xdot + p a_2 \xdot^2 + \cdots + p^{n} a_{n+1} \xdot^{n+1} \mod p^{n+1}$ by induction. 
We have proven the base case and proceed to solve for $\ydot$ in terms of $\xdot$ as we did before inductively. 
As before we have 
$$\delta( f(x,y) ) = \frac{1}{p} \left(  \fphi(x^p + p \xdot, y^p + p \ydot) - f(x,y)^p \right)=0.$$ We use the expansion 
 $$\fphi(x^p + p\xdot, y^p + p \ydot) = \sum_{d\geq 0} p^{d-1} h_d(\xdot,\ydot)$$ 
where $h_d$ are homogeneous polynomials of degree $d$ in $\xdot$ and $\ydot$ with coefficients in $R[x,y]/(f)$; this gives
\begin{equation}\label{eqn:aff}
 \frac{\fphi(x^p,y^p) - f(x,y)^p}{p} + \sum_{d=1}^n p^{d-1} h_d(\xdot,\ydot) \equiv 0 \mod p^{n+1}.
\end{equation}
By inductive hypothesis we may assume $\ydot = A + p^n B$ where $A = a_0 + \sum_{j=1}^n p^{j-1} a_j \xdot^j$. Expanding the homogeneous polynomials gives
  $$h_d(\xdot, \ydot) = h_d(\xdot,A + p^n B)= h_d(\xdot, A) + \frac{\partial h_d}{\partial \ydot}(\xdot,A) p^n B \mod p^{n+1}$$
and substituting into equation \ref{eqn:aff} we get 
\begin{equation}\label{eqn:aff2}
r + \sum_{d=1}^n p^{d-1} \left( h_d(\xdot ,A) + \frac{\partial h_d}{\partial \ydot}(x,A) p^n B \right)
= r + \sum_{d=1}^n p^{d-1} h_d(\xdot,A) + \sum_{d=1}^n p^{d-1} \frac{\partial h_d}{\partial \ydot}(\xdot,A) p^n B
\end{equation}
where $r = \frac{\fphi(x^p,y^p) - f(x,y)^p}{p}$. Note that the left terms on the right side of equation \ref{eqn:aff2} can be written as 
$$r + \sum_{d=1}^n p^{d-1} h_d(\xdot,A) = p^n C$$
and the term on the right can be written as
$$\sum_{d=1}^n p^{d-1} \frac{\partial h_d}{\partial \ydot}(\xdot,A) p^n B \equiv \frac{ \partial h_1}{ \partial \ydot}(\xdot,A) p^n B \mod p^{n+1}.$$
Using the fact that $h_1 = \fphi_x(x^p,y^p) \xdot + \fphi_y(x^p,y^p) \ydot$ we have $\frac{\partial h_1}{\partial \ydot}(x,A) = \fphi_y(x^p,y^p)$ which tells us that $p^n C + \fphi_y(x^p,y^p) p^n B \equiv 0 \mod p^{n+1}$
and hence that $C + \fphi_y(x^p,y^p) B \equiv 0 \mod p$ and finally that
 $$B = -C /\fphi_y \mod p.$$
It remains to show that $B$ has degree less than or equal to $n$ in $\xdot$. 

We note that $p^n C = r + \sum_{d=1}^{n+1} p^{j-1} h_d(\xdot,A) \mod p^{n+1}$ where we can write $h_d(S,T) = \sum_{j+k= d} a_{j,k}^d S^j T^k$, where $a^d_{j.k}\in R[S,T]/(f)$. We can expand the expression 

\begin{equation}\label{eqn:aff3}
p^{d-1} h_d(\xdot, A) = p^{d-1} h_d(\xdot, a_0 + a_1 \xdot + \cdots + p^{n-2} a_{n-1} \xdot^{n-1})
\end{equation}
so that its general term takes the form
$$p^{d-1} a_{i,j}^d \xdot^i ( a_0 + a_1 \xdot + p a_2 \xdot^2 + \cdots + p^{n-2} a_{n-1} \xdot^{n-1})^j.$$
We expand this general term further to get
\begin{eqnarray*}
&& (a_0 + a_1 \xdot + p a_2 \xdot^2 + \cdots + p^{n-2} a_{n-1} \xdot^{n-1})^j \\
&=& \sum_{j_0 + j_1 + \cdots + j_{n-1} = j} a_0^{j_0} (a_1 \xdot)^{j_1} (p a_2 \xdot^2)^{j_2} \cdots (p^{n-2} a_{n-1} \xdot^{n-1} )^{j_{n-1}} \\
&=&\sum_{j_0 + j_1 + \cdots + j_{n-1} = j} a_0^{j_0} a_1^{j_1} a_2^{j_2} \cdots a_{n-1}^{j_{n-1}} p^{j_2 + 2 j_3 + 3 j_4 + \cdots + (n-2)j_{n-1}} \xdot^{j_1 + 2 j_2 + 3 j_3 + \cdots + (n-1)j_{n-1}}
\end{eqnarray*}
So that a general term of equation \ref{eqn:aff3} takes the form
$$\alpha p^a \xdot^b$$
where $\alpha \in \OO(U)$ and 
\begin{eqnarray*}
i+j &=& d\\
a &=& d-1 + j_2 + 2 j_3 + \cdots + (n-2) j_{n-1}\\
b &=& i + j_1 + 2j_2 + \cdots + (n-1)j_{n-1}\\
j &=& j_0 + j_1 + \cdots + j_{n-1} 
\end{eqnarray*}
Using these relations we show
\begin{eqnarray*}
a &=& d-1 + j_2 + 2 j_3 + \cdots + (n-2) j_{n-1}\\
&=& i+j -1 + j_2 + 2 j_3 + \cdots + (n-2) j_{n-1}\\
&=& i-1 + j_0 + j_1 + 2j_2 + 3 j_3 + \cdots + (n-1) j_{n-1}\\
&=& i-1 + j_0 + (b-i)\\
&=& b-1 + j_0 
\end{eqnarray*}
Which tells us the $a=b-1 + j_0 \geq b-1$. 
Notice that the degree of the general term is $b$ and we want to show that $b\leq n+1$. Suppose this is not the case and that $b>n+1$. 
This implies that $a >n$ which implies $\alpha p^a \xdot^b \equiv 0 \mod p^{n+1}$; 
so such a term doesn't contribute to $\ydot \mod p^{n+1}$. This concludes the proof.

We will now prove the second part of the theorem.
Let $f \in R[S,T]$, and write $ f(S,T) = \sum_{k=0}^d f_k(S,T)$ where $f_k$ homogeneous of degree $d$ i.e. $f_0 = a_{00}$, $f_1 = a_{10} S + a_{01} T$, $f_2 = a_{20} S^2 + a_{11} ST + a_{02} T^2$ and so on. 
We have $f_d\neq 0$ since $f$ is of degree $d$

Using this decomposition we can compute the partial derivatives term-wise to get 
 $$ \frac{\partial f}{\partial S} = \sum_{k=1}^d \frac{\partial f_k}{\partial S}, \ \ \ \  \frac{\partial f}{\partial T} = \sum_{k=1}^d \frac{\partial f_k}{\partial T}.$$ 
 
 If $\frac{\partial f }{\partial S} \equiv \frac{\partial f}{\partial T} \equiv 0 \mod p$ identically then $$S \frac{\partial f}{\partial S} + T \frac{\partial f}{\partial T} = \sum_{k=1}^d \left( S \frac{\partial f_k}{\partial S} + T \frac{\partial f_k}{\partial T}\right)  = \sum_{k=1}^d k f_k \equiv 0 \mod p$$ and since $R_0[S,T] \equiv \bigoplus_{k\geq 0} (R_0[S,T])_k$ we must have that $k f_k(S,T) \equiv 0 \mod p$ for $k=1,\ldots,d$. 
 If $p\nmid k$ this means that $f_k(S,T)=0$ which tells us that
  $$ f(S,T) = h(S^p,T^p) + p g(S,T) .$$
Note in particular that 
 $$ \frac{\partial f}{\partial S} \equiv \frac{\partial f }{\partial T} \equiv 0 \mod p \implies \deg(f) \geq p. $$
\end{proof}

\begin{lem}\label{lem: etale property}
Let $X \subset \PP^n$ be a smooth projective curve. 
Suppose $\Lambda$ and $\Lambda'$ are complementary linear subspaces of $\PP^n$. 
 $\pi = \pi^{\Lambda'}_{\Lambda}: X \setminus (X \cap \Lambda') \to \Lambda$ is \'etale at $x \in X$ if and only if $\overline{ x, \pi(x)} \neq T_{X,x}$.
 
 If $X = V(f(x,y))$ is an affine plane curve, the projection to the $x$-axis is \`etale if and only if $\partial f/\partial y \neq 0$.
 Similarly for projections to the $y$-axis.
\end{lem}
\begin{proof}[Proof]
 By change of coordinates and by localness of the problem one only needs to consider projections $\pi:\AA^n_R \to \AA^r_R$ defined by $\pi(x_1,\ldots,x_r,\ldots, x_n) = (x_1,\ldots,x_r)$ and curves of the form $X = \Spec \ R[x_1,\ldots,x_n]/(f_1,\ldots,f_e)$.
 
 Let $a$ be a point of $X$ not in $\Lambda$. 
 The lines of projection $\overline{a,\pi(a)}$ are the unique lines connecting the $a$ and $\pi(a)$ which one can compute explicitly.
 
 Let $J(a)$ be the jacobian of $f = (f_1,\ldots, f_e)$ with respect to the variables $(x_{r+1},\ldots, x_{n})$ evaluated at $a$.
 
 We use the following two facts:
   \begin{enumerate}
    \item \label{first} The condition on $\pi$ being \`etale is equivalent to the $J(a)$ having maximal rank.
    \item \label{second}  The condition that $\overline{a,\pi(a)} \subset T_{X,a}$ is equivalent to $J(a) \cdot \left[ \begin{matrix} a_{r+1} \\ \vdots \\ a_n \end{matrix} \right] =0.$
   \end{enumerate}

 Suppose that $\pi$ is \`etale at $a \in X$. 
 By the property \ref{first}, $J(a)$ has full rank.
 This implies there exists a left inverse $K$ such that $K\cdot J(a)$ is the $n-r\times n-r$ identity matrix.
 The existence of such a $K$ contradicts $\overline{a ,\pi(a)} \subset T_{X,a}$ in view of property
 \ref{second}.
 
 Conversely suppose that $\overline{a,\pi(a)}$ is not contained in $T_{X,a}$.
 This is equivalent to 
 $$J(a) \cdot \left[ \begin{matrix} a_{r+1} \\ \vdots \\ a_n \end{matrix} \right] \neq 0,$$
 by property \ref{second}.
 This implies that $J(a)$ has rank at least one.
 Since $J(a)$ has rank at most one it has full rank which is equivalent to \`etaleness by property \ref{first}.
 
 The second property is a special case of the first. 
\end{proof}

Figure \ref{fig:etale projection from line} shows the projection from a line to another line.

\begin{figure}[h] 
\begin{center}
 \includegraphics[scale=0.5]{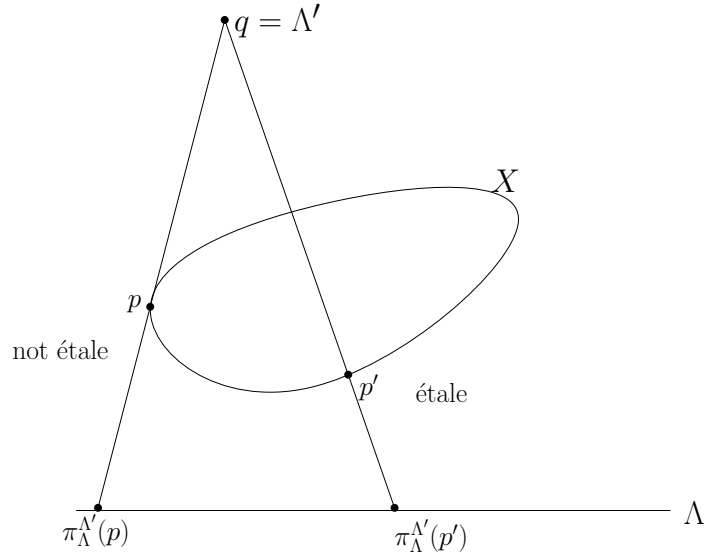}
\end{center} 
\caption{ A projection in to $\Lambda$ with center $\Lambda'$.  \label{fig:etale projection from line}}

\end{figure}

\begin{lem}\label{lem: injectivity}
  Let $X\subset \PP^n_R$ be a smooth irreducible curve of degree $d<p$.
  Let $\pi_1$ and $\pi_2$ be projections onto lines in $\PP^2_{R} \subset \PP^n_R$ where the centers of projections do not intersect $X$.
  
  Let $\vareps_1,\vareps_2:U \to \AA^1_R$ be restrictions of $\pi_1$ and $\pi_2$ so that they are both \'{e}tale onto their image.
  
  \begin{enumerate}
   \item The map $\sigma := (\eps_1 \times \eps_2)^*: R[S,T] \to \OO(U)$ has the property that 
   the induced map $ \sigma_0: R_0[S,T]/(\fbar) \to \OO(U)/p$ is injective.
 
   \item Let $\psi_1,\psi_2: J^1(U) \isom \AAhat^1 \hattimes \Uhat$ denote the affine bundle trivializations associated to $\vareps_1$ and $\vareps_2$ respectively. 
   For every $n\geq 1$ we have
 $$ \psi_{21} \otimes_R R_n \in A_{n}.$$ 
  \end{enumerate}
\end{lem}

\begin{figure}[h]
\begin{center}
 \includegraphics[scale=0.5]{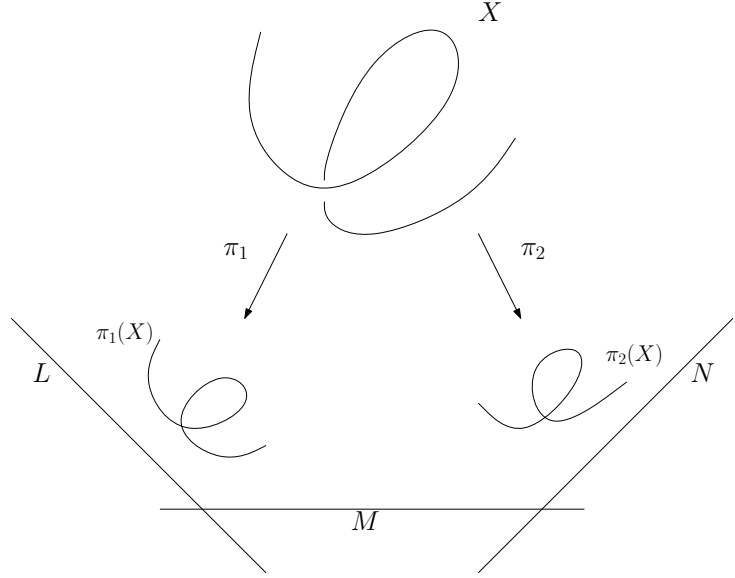}
\end{center} 
\caption{ A curve $X \subset \PP^N$ with two projections onto $\Lambda_1 = \overline{L,M}$ and $\Lambda_2 = \overline{M,N}$ both isomorphic to $\PP^2$. 
The \'{e}tale projections $\eps_L,\eps_M$ and $\eps_N$ to the lines $L,M$ and $N$ which induce the trivializations on the $J^1(X)$ factor through the projections $\pi_1$ and $\pi_2$ \label{fig:projections}}
\end{figure}

\begin{proof}

In what follows an overline will denote a Zariski closure.

Let $\vareps_1^*(T) = x$ and $\vareps_2^*(T)=y$ where $T$ is the \'etale parameter on $\AA^1$. 
Define $\sigma: R[S,T]\to \OO(U):=B$ by $S \mapsto x$ and $T\mapsto y$. 
Since the image of $\sigma$ is an integral domain we know that $\ker(\sigma)$ is a prime ideal.
Since $R[S,T]$ is a UFD and the $\ker(\sigma)$ has height 1 we know that there exists some irreducible $f \in R[S,T]$ such that $\ker(\sigma) = (f)$. 
This $f$ is the minimal relation among $x$ and $y$ and we have the equation $f(x,y)=0$. 
Geometrically we have 
$$\overline{\eps_1 \times \eps_2(U)}= V(f) \subset \AA^2,$$ 
where $f$ is a dehomogenization of $F$ where $F$ defines $\pi(X)=V(F) \subset \PP^2$.
We know that $f$ is irreducible by topological considerations. 
Note that the image is not necessarily non-singular or even flat.

We will now show that $\overline{\pi_0(U_0)} = \overline{ \pi(U) }_0$ by demonstrating a closed immersion $\overline{\pi_0(U_0)} \subset (\overline{\pi(U)})_0$ and $\deg(\overline{\pi_0(U_0)})= \deg(\overline{ \pi(U) }_0).$\footnote{ Let $X = X_1 \cup \ldots \cup X_r$ is a decomposition into irreducible components
and write $X_i = V(f_i)$ where $f_i$ is an irreducible polynomial. 
This implies $X=V(f)$ where $f = \prod_{i=1}^r f_i$.
This implies $\deg(X) \geq \deg(X_i)$.
We have $\deg(X) \geq \deg(X_i)$. 

Note that if $\deg(X) = \deg(X_i)$ then $X=X_i$. 
This is because $f_i \vert f$ and $\deg(f_i) = \deg(f)$ implies $\deg(f/f_i) = 0$ 
which implies $(f) = (f_i)$.}

Let $J \subset R[S,T]$ be the ideal defining $\overline{\pi_0(U_0)} \subset \AA^2_{R}$.
By commutativity of
$$
\xymatrix{
  R[S,T] \ar[r]^{\sigma} \ar[d]^{\alpha} & B \ar[d] \\
  R[S,T]/p \ar[r]^{\sigma_0} & B/p.
}
$$
we have $(f,p) \subset \ker(\alpha \circ \sigma_0)=J$.
This implies $\overline{\pi_0(U_0)}= V(J) \subset V(f,p) = (\overline{\pi(U)})_0\subset \AA^2_R$. 

Observe $\deg( \overline{ \pi_0(U_0)}) = \deg( \pi_0(X_0) ) = \deg(X_0) = d$.
On the other hand $\deg( \overline{ \pi(U)}_0) = \deg( \pi(X_0) ) = \deg( F \mod p) \leq d = \deg( \overline{ \pi_0(U_0)})$.
We can now conclude that $(f,p)=J$. 

This implies that $\ker(\sigma_0) =J/(p) = (\fbar)$. 
Since $A_0/(\fbar) = A_0/\ker(\sigma_0) \inclusion \OO(U_0)$
We can work directly with the equation $f(x,y)=0$. 
In particular we use nonvanishing of $\partial f/\partial x$ and $\partial f/ \partial y$ which follows from the description of \`etale projections (Lemma \ref{lem: etale property})
\begin{eqnarray*}
  && \r + \fphi_x(x^p,y^p) \xdot + \fphi_y(x^p,y^p) \ydot \\ && + \frac{p}{2}( \fphi_{xx}(x^p,y^p) \xdot^2 + 2 \fphi_{xy}(x^p,y^p) \xdot \ydot + \fphi_{yy}(x^p,y^p) \ydot^2 \equiv 0 \mod p^2
\end{eqnarray*}
where $\r = \frac{\fphi(x^p,y^p)-f(x,y)^p}{p} \in \OO(U)$. 

Hence $\psi_{21}$ can be computed by solving for either $\xdot$ in terms of $\ydot$ or $\ydot$ in terms of $\xdot$. 
This is possible mod $p^n$ for every $n\geq 2$ if either $f_x(x^p,y^p)$ or $f_y(x^p,y^p)$ is invertible in $\OO(U)_0$. 
This is equivalent to having $f_x$ or $f_y$ being not identically zero mod $p$ and the projections are \`etale on $U$ exactly when the partial derivatives are nonvanishing. 
This is true since the morphisms $\sigma_0: R_0[S,T]/(f) \to \OO(U)/p$ is injective (which we just proved).
We now apply the local computations (Lemmm \ref{lem:non-vanishing criteria} to establish 
$$\psi_{21} \mod p^{n} \in A_n$$ 
for each $n\geq 2$.
\end{proof}

\begin{figure}[h]\label{fig:synthetic PP3}
\begin{center}
 \includegraphics[scale=0.5]{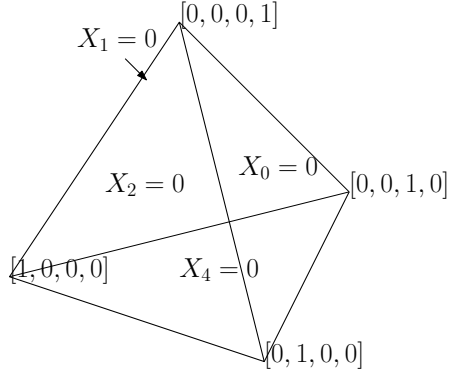}
\end{center} 
\caption{ A picture of $\PP^3$ with its standard decomposition.}
\end{figure}

\begin{figure}[h]\label{fig:synthetic PP3 bad}
\begin{center}
 \includegraphics[scale=0.5]{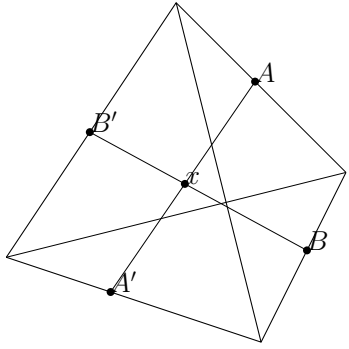}
\end{center} 
\caption{ If there exists some point $x$ such that $T_{x}X$ is equal $L(\Lambda',\Lambda,P)$ for all $\Lambda\in \lambda_2$ then we would have $\overline{AA'} = \overline{B'B}$. The situation looks very bad in this simple case.}
\end{figure}

\begin{lem}\label{lem: covers}
 Let $R = W_{p,\infty}(k)$ where $k\subset \FFpbar$.
 Let $X \subset \PP^n_R$ be a smooth irreducible curve. 
 There exists a system of linear forms $\lambda = \lbrace \l_0 , \ldots, \l_n \rbrace$ 
 such that $ (X_{\Lambda} \to X)_{\Lambda \in \lambda_2}$
 form a cover and which $\pi_{\Lambda}: X_{\Lambda} \to \AA^1_R \subset \PP^1_R$ \`etale onto its image.
 (cf section \ref{projs}).
\end{lem}
\begin{proof}
It suffice to show that there exists a decomposition $\lambda$ over $\FFpbar$ since we can lift any such decomposition to $R$.

Suppose in addition that $X\subset \PP^N$ is a curve and that for all $\Lambda' \in \lambda_{N+1-2} \cup \lambda_{N+1-3}$ we have $X \cap \Lambda' =\emptyset$ so that all of the projections 
 $$\pi^{\Lambda'}_{\Lambda}: X\to \Lambda\cong \PP^1 \mbox{ or } \PP^2$$ 
are well-defined.
Without loss of generality we can assume that the decomposition $\lambda$ comes from the coordinates $X_0,\ldots,X_N$ on $\PP^N$.

Suppose that there exists some $x \in X$ such that for all $\Lambda \in \lambda_2$ that $\pi^{\Lambda'}_{\Lambda}(x)$ is not \'etale at $x$. Using the notation introduced in equation \ref{eqn:unique line} we would have
 $$ L(\Lambda',\Lambda,x) = T_xX$$
for all $x\in X$. Here $T_xX$ is interpreted as the physical tangent line for the embedded curve $X$. This leads a silly situation which we will show cannot be possible by means of synthetic argument. 
See figures \ref{fig:synthetic PP3} and \ref{fig:synthetic PP3 bad} for a picture of this situation.

Suppose that $M,N\in \lambda_2$ are not equal and let $L=L(M,M',x)$ and $K=L(N,N',x)$.
We claim that not both $L$ and $K$ can be in the tangent space of $x$.

Let $A$ be the unique point where $L$ intersects $M$ and $A'$ be the unique point where $L$ intersects $M'$. 
Define $B$ and $B'$ similarly for $N$ and $N'$.
If both $L$ and $K$ are lines tangent to $X$ at $x\notin \lambda_1$ we have $L = K$. 
This implies $L$ intersects $M$ at $A$.
This also implies $L$ also intersects $N$ at $B$. 
But $M$ and $N$ intersect in a unique point $C$. 
This means that $M$, $N$ and $L$ are contained in the unique plane $\pi$ spanned by $A$, $B$ and $C$. Since $\pi$ is also the unique plane spanned by $M$ and $N$, this means that $\pi \in \lambda_3$. 
But by hypothesis we supposed that $x$ was not in any $\pi \in \lambda_{3}$ which is a contradiction.

It remains to show that for every curve $X\subset\PP^N_{\FFpbar}$ there exists some decomposition $\lambda$ such that $X$ does not intersect any $\Lambda' \in \lambda_{N+1-3}$. 
This can be done by the moving lemma and dimension counting.

Recall that if $X$ and $W$ are subvarieties of $\PP^N$ we say they intersect properly if
 $$\dim(X\cap W ) = \max\{ \dim(W) + \dim(X) - N, 0\}.$$

Let $W$ be the unions of the centers of projections to coordinate planes. 
$ W= \bigcup_{\Lambda \in \lambda_{3} } \Lambda'.$ 
Since $W$ has dimension $N-1-2$ and $X$ has dimension $1$ if $W$ and $X$ intersected properly we would have 
$$\dim(X\cap W) = (N-1-2) + 1 - N = -2$$
which imply that the intersection is empty. 
By the moving lemma ($\FFpbar$ is an infinite field) we can arrange so that $X$ and $W$ have an empty intersection.
\end{proof}

\begin{thm}\label{an existence}
 Let $R = W_{p,\infty}(k)$ where $k = \FFbar_p$. 
 Let $X_d \subset \PP^N_R$ be a smooth irreducible curve of degree $d$ and suppose that $d>p$ then for every $n\geq 1$, 
 $\J^1(X)_n \to X_n$ admits an $A_{n}$-structure.
\end{thm}
\begin{proof}
Let $\lambda = \lbrace l_0, \ldots, l_n\rbrace $ as in Lemma \ref{lem: covers}.
By change in coordinates we can assume without loss of generality that the $l_0,\ldots, l_n$ are the coordinate hyperplanes given by
$l_i = V(X_i)$.
Let $l_0', \ldots, l_n'$ be the coordinate axes,
Let $U_i$ be the subset of $X$ where the projection map to $l_i'$ is \`etale
and $\vareps_i: U_i \to \AA^1$ be the \`etale projection.
and let $\psi_i: \J^1(U_i) \to \Uhat_i \hattimes \AAhat^1_R$ be the affine bundle chart of $\Jhat^1(X)$ associated to $\vareps_i$.

For each pair of lines $\Lambda_1,\Lambda_2 \in \lambda_1$ one can see that $\pi_{\Lambda_1}$ and $\pi_{\Lambda_2}$ factor through $\pi_{\Lambda}$ where $\Lambda = \overline{\Lambda_1,\Lambda_2}$.
Letting $U = X_{\Lambda_1} \cap X_{\Lambda_2}$ puts us in the hypotheses of Lemma \ref{lem: injectivity}.
If $\psi_1$ and $\psi_2$ are the associated transition maps we have 
$\psi_{12} = (\psi_1 \circ \psi_2^{-1}) \otimes_R R_n \in A_{n}(U_{ij})$ for every $n\geq 1$ which proves our result.
\end{proof}

\subsection{Steps 3 and 4: Reduction of an $A_{n}$-structure}\label{sec: step34}

The following theorem allows us to reduce the structure group of the first $p$-jet space of a smooth curve $X/R$ of genus $g\geq 2$.

\begin{thm}\label{reduction}
Let $X/R$ be a scheme and $\pi: E\to X$ an $\AA^1_{R}$-bundle. 
Suppose that $E_n/X_n$ admits $A_{n}$-structures.
Let $[L_0] \in \Pic(X_0)$ be the class naturally associated to the $\AL_1(\OO_{X_0})$-structure on $E_0$ as in Remark \ref{classes}.

If $H^1(X_0,L_0^*)=0$ then $E_n$ admits an $\AL_1(\OO_{X_n})$-structure. 

Here $L_0^*$ denotes the dual of $L_0$.
\end{thm}
\begin{proof}
Let $\psi_{ij}^{(n)} \in A_{n+1}(U_{ij})$ be the transition maps on a trivializing cover for $E_n$. 
We will prove that $\psi_{ij}^{(n)} \sim_{A_{n+1}} \psitilde_{ij}^{(n)} \in \AL_1(\OO_{X_n})$ by induction.

The base case with $n=0$ is trivial since $\Auts(\AA^1_{R_n}) = \AL_1(\OO_{X_0})$. 

We will suppose now that $\psi_{ij}^{(n-1)} \in \AL_1( \OO_{X_{n-1}})$ and construct some $\psi_i$'s in $A_{n+1}(U_i)$ such that
 $$ \psi_i \psi_{ij}^{(n)} \psi_j ^{-1} \in \AL_1(\OO_{X_n}). $$

Let $2 \leq r \leq n+1$ and define $M_{n,r}\leq \Auts(\AA^1_{R_n})$ so be the automorphisms of degree less than $r$ of the form
 $$ \psi = a_0 + a_1 T  + p^n( b_2 T^2 + \cdots + b_r T^r) \mod p^{n+1}.$$
Note that $\psi_{ij}^{(n)} \in M_{n,n+1}$ since $\psi_{ij}^{(n-1)} \in \AL_1(\OO_{X_{n-1}})$ and $\psi_{ij}^{(n)} \equiv \psi_{ij}^{(n-1)} \mod p^n$.

We show now prove the following claim: For every $r\geq 2$ if $\psi_{ij}^{(n)} \in M_{n,r}$ then there exists some $\psi_{ij}{'}^{(n)} \in M_{n,r-1}$ such that 
 $$ \psi_{ij}^{(n)} \sim_{M_{n,r}} \psi_{ij}{'}^{(n)} \mbox{   and   } \psi_{ij}^{(n+1)} \equiv \psi_{ij}{'}^{(n)} \mod p^{n+1}.$$
(Note that when we get to $r=2$ we will have shown the structure group on $E_n$ can be reduced to $\AL_1(\OO_{X_n})$.)

For $r\geq 2$ define $\tau_r: M_{n,r} \to \OO_{X_0}$ by 
$$ \tau_r(\psi)= \frac{b_r(\psi)}{a_1(\psi)} \mod p.$$
Now if $\psitilde = \atilde_0 + \atilde_1 T + p^n( \btilde_2 T^2 + \cdots + \btilde_r T^r) \in M_{n,r}$ is another element we have
$$ \tau_r( \psi \circ \psitilde) = \frac{a_1 \btilde_r + b_r \atilde_1}{\atilde_1 a_1} = \tau_r(\psi) \atilde_1^{r-1} + \tau_r(\psitilde).$$ 
This shows $\tau_r$ is a group cocycle with respect to the action of $M_{n,r}$ on $\OO_{X_0}$ (which factors through the quotient $M_{n,r} \to \AL_1(\OO_{X_0}) = \OO_{X_0} \rtimes \OO_{X_0}^{\times} \to \OO_{X_0}^{\times}$, and $\OO_{X_0}^{\times}$ acts on $\OO_{X_0}$ via multiplication after raising an element to the $(r-1)$-st power.

The group cocycle $\tau_r$ induces a group homomorphism $\sigma_r: M_{n,r} \to \OO_{X_0}^{\times} \ltimes \OO_{X_0}$ given by
 $$ \sigma_r: \psi \mapsto ( a_1(\psi)^{r-1} \mod p, \tau_r(\psi) ).$$
Note that this is indeed a group homomorphism:
$$( a_1^{r-1} , \tau_r(\psi)) * (\atilde_1^{r-1}, \tau_r(\psitilde)) = (a_1^{r-1}\atilde_1^{r-1}, \tau_r(\psi) \atilde_1^{r-1} + \tau_r(\psitilde)) = ( (a_1\atilde_1)^{r-1}, \tau_r(\psi\circ \psitilde) ).$$

Let $(m_{ij},a_{ij})$ be the image of the cocycle $\psi_{ij}^{(n)}$ under the map $\sigma_r$. 
Note that 
$$ (1, 0)=(m_{ij}, a_{ij})(m_{jk},a_{jk})(m_{ki},a_{ki}) = (m_{ij}m_{jk} m_{ki})(a_{ij}m_{jk}m_{ki} + a_{kj}m_{ki} + a_{ki}) $$

The condition on the $a_{ij}$'s is a really a condition for a cocycle with values in line bundles: 
Let $L_0$ is a line bundle on $X_0$ with trivializations
 $$ L_0(U_i) = \OO(U_i) v_i$$
where 
 $$ v_j = m_{ij}v_i.$$
Suppose $s_{ij} \in L_0(U_{ij})$ defines a cocycle and define $a_{ij}$ by
 $$ s_{ij} = a_{ij}v_j.$$
Then we have 
 $$ 0=s_{ij} + s_{jk} + s_{ki} = a_{ij} m_{ij} v_i + a_{jk} m_{ik} v_i + a_{ki} v_i.$$

If follows then that since $(m_{ij},a_{ij})) \in \OO_{X_0}^{\times} \ltimes \OO_{X_0}$ define a cocycle then the collection
 $$ s_{ij} := a_{ij} v_j^* \in L^*(U_{ij}) $$
define a cocycle in $L^*$. 

By hypothesis $H^1(X_0, L_0^*)$ is trivial (a simple Riemann-Roch computation) and we have 
 $$ s_{ij} = s_i - s_j$$
for some collection $s_i \in L_0^*(U_i)$. 
Define $a_i \in \OO_{X_0}(U_i)$ by
 $$ s_i = a_i v_i^*, \ \ \  i \in I.$$
This gives 
 $$ s_i - s_j = (a_i m_{ij} a_k) v_j^* $$
so 
$$ a_{ij} = a_i m_{ij} - a_j$$
or 
$$ -a_i m_{ij} + a_{ij} + a_j =0.$$

In terms of \v{C}ech cochains on $\OO_{X_0}^{\times} \ltimes \OO_{X_0}$ this means
$$ (1,a_i)(m_{ij},a_{ij})(1,a_j) = (m_{ij}, -a_im_{ij}+a_{ij}+a_j) = (m_{ij},0). $$

Let $\psi_i = T - p^n a_i T^t$ be elements of $M_{n,r}(U_i)$. 
We have 
$$ \sigma_r ( \psi_i \circ \psi_{ij} \circ \psi_j^{-1} ) = (m_{ij},0) $$
which implies that $\psi_i \circ \psi_{ij} \circ \psi_j \in M_{n,r-1}$. 
Hence we have that for all $r\geq 2$ and all $\psi_{ij}^{(n)} \in M_{n,r}$ there exists some $\psi_{ij}^{(n)} \in M_{n,r-1}$ such that $\psi_{ij}^{(n)} \sim_{M_{n,r}} \psi_{ij}'{}^{(n)}$ and $\psi_{ij}^{(n+1)} \equiv \psi_{ij}'{}^{(n)} \mod p^{n+1}$.
This completes the proof. 
\end{proof}

To complete the proof of the Main theorem we apply Theorem \ref{reduction} to $E$ being the first $p$-jet space of a curve together with its $A_n$-structure given in section \ref{sec: an structs}

\begin{proof}[Proof of Main Theorem]
Consider the $A_n$-structure on $J_p^1(X)_n$ coming from Theorem \ref{an existence}.
We apply theorem \ref{reduction} so that $L_0 = \FT_{X_0}$ -- by Riemann-Roch we have $H^1(X_0,L_0^*)=0$ are we are in the hypotheses of Theorem \ref{reduction}.
\end{proof}

\bibliographystyle{alpha}
\bibliography{../../bib/clean}

\end{document}